\documentclass[10pt,twoside]{amsart}
\usepackage{amsmath}
\usepackage{amssymb}
\usepackage[curve]{xypic}

\usepackage[colorlinks=true, urlcolor=blue,bookmarks=true,bookmarksopen=true,
citecolor=blue,hypertex]{hyperref}

\numberwithin{equation}{section}

\newtheorem{theorem}{Theorem}[section]
\newtheorem{defn}[theorem]{Definition}

\newtheorem{corollary}[theorem]{Corollary}

\newtheorem{figger}[theorem]{Figure}
\newtheorem{lemma}[theorem]{Lemma}
\newtheorem{prop}[theorem]{Proposition}
\newtheorem{remark}[theorem]{Remark}

\def \begineq{\begin{equation}}
\def \endeq{\end{equation}}

\def \bb{\mathbb}
\def \mb{\mathbf}
\def \mc{\mathcal}

\def \CC{{\bb{C}}}

\def \QQ{{\bb{Q}}}
\def \RR{{\bb{R}}}

\def \ZZ{{\bb{Z}}}

\def \GGC{{\mc G}}

\def \MMC{{\mc M}}

\def \PPC{{\mc P}}

\def \({\left(}
\def \){\right)}
\def \<{\langle}
\def \>{\rangle}

\def \bar{\overline}
\def \deg{\mathrm{deg}}

\def \inter{\cap}
\def \into{\hookrightarrow}

\def \tensor{\otimes}
\def \union{\cup}
\def \vargeq{\geqslant}
\def \varleq{\leqslant}
\def \xto{\xrightarrow}

\def \Aut{{\rm Aut}}

\def \Diff{{\rm Diff}}

\def \Ham{{\rm Ham}}

\def \id{{\rm id}}
\def \img{{\rm img }}

\def \Span{{\rm Span}}

\def\nat{\Psi_L}
\def\comp{\mathrm{comp}}
\def\cm{MC} \def\cf{FC} \def\hm{MH} \def\hf{FH}
\def\pss{PSS}
\usepackage{graphicx}
\def\co{\colon\thinspace}

\def \qed{\hfill $\square$ \vspace{0.03in}}

\begin{document}

\title{Homological Lagrangian Monodromy}

\author{Shengda Hu}
\address{Department of Pure Mathematics, University of Waterloo, 200 University Ave. West, Waterloo, Canada}
\email{hshengda@math.uwateloo.ca}

\author{Fran\c{c}ois Lalonde}
\address{D\'epartement de math\'ematiques et de Statistique, Universit\'e de Montr\'eal, C.P. 6128, Succ. Centre-ville, Montr\'eal H3C 3J7, Qu\'ebec, Canada}
\email{lalonde@dms.umontreal.ca}

\author{R\'emi Leclercq}
\address{D\'epartement de math\'ematiques et de Statistique, Universit\'e de Montr\'eal, C.P. 6128, Succ. Centre-ville, Montr\'eal H3C 3J7, Qu\'ebec, Canada}
\email{leclercq@dms.umontreal.ca}

\maketitle

\abstract  We show that the Hamiltonian Lagrangian monodromy group, in its homological version, is trivial for any weakly exact Lagrangian submanifold of a symplectic manifold. The proof relies on a  sheaf approach to Floer homology given by a relative Seidel morphism.  \endabstract

\footnote{2010 Mathematics Subject Classification 53C15, 53D12, 53D40, 53D45, 57R58, 57S05, 58B20.}


\section{Introduction}

 Given a Lagrangian submanifod $L \subset M$ embedded in a symplectic manifold $M$, it is natural to consider the   subgroup $\GGC \subset \Diff (L)$ consisting of all diffeomorphisms of $L$ that can be obtained as the time-one map of a Hamiltonian (i.e. exact) Lagrangian isotopy $\phi_{t \in [0,1]}: L \to M$ that starts at $t=0$ at the identity map on $L$ and ends at $t=1$ at a diffeomorphism that preserves $L$. In other words, if one denotes by $\Ham_L(M) \subset \Ham(M)$ the subgroup of the group  of Hamiltonian diffeomorphisms of $M$ consisting of the diffeomorphisms $g$ sastisfying $g(L) = L$, the group $\GGC$ is  then the image of the homomorphism
$$
\Ham_L(M) \to \Diff (L)
$$
that assigns to each diffeomorphism $g \in \Ham_L(M)$ its restriction to $L$.  Denoting by $R$ any given ring, the {\em homological Hamiltonian Lagrangian monodromy problem} is the study of the subgroup $\GGC_{*,R}$ of $\Aut(H_*(L;R))$ defined as the image of $\GGC$ under the map that assigns to each diffeomorphism its action on homology (we will often assume that $R$ is given and will omit it in our notations; we will also omit the word ``exact'' since we will work with Hamiltonian isotopies only in this paper).

\medskip
The homological Lagrangian monodromy group  $\GGC_{*,R}$ is an invariant attached to each exact Lagrange isotopy class of a given Lagrangian submanifold. It is therefore of prime importance in the Lagrangian knot problem.

\medskip
To our knowledge, this group has been studied only very recently by  Mei-Lin Yau \cite{Yau},  in the two cases of the standard monotone $2$-torus and of the Chekanov torus, both living in $\RR^4$, using soft methods in a clever way. Let $\Theta_{t\in[0,1]}$ be the standard one-parameter family of elements of $SO(2)$ starting at the identity anti-clockwise and ending at the rotation by angle $\pi/2$. By the standard inclusion $SO(2) \subset U(2)$, the same path can be considered as a path of elements of $U(2)$ and it clearly restricts  to an exact isotopy of the standard torus $T_{a,a} = S^1(a) \times S^1(a) \subset \CC \times \CC$ (here the number in parentheses denotes the area of the circle) whose endpoint at $t=1$ permutes the two standard generators of $H_1(T_{a,a}; \ZZ)$. The main result of \cite{Yau} is that this induces the only non-trivial element of $\GGC_{*,\ZZ}$, and thus $\GGC_{*,\ZZ} = \ZZ_2$ for the standard torus.  M.-L.Yau also shows that the same result holds for the Chekanov 2-torus of $\RR^4$. Moreover she shows that the intrinsic spectral (and symplectic) invariants attached to the non-trivial element of $\GGC_{*,\ZZ}$ in each of these two cases are different, so that this provides another proof of the fact that the standard 2-torus is not exact Lagrange isotopic to the Chekanov 2-torus.

\medskip
 The other extreme case is the one of a closed exact Lagrangian submanifold $L$ in a cotangent bundle $T^*V$. A famous conjecture states that $L$ should then be Hamiltonian isotopic to the zero section. If this conjecture is true, then obviously the group $\GGC_*$ is trivial for all coefficients, that is to say it consists of the identity only. A homological version of this conjecture has been proved by Fukaya, Seidel and Smith in \cite{FSS}:  they have indeed shown that  if $V$ is  simply connected, then an exact Lagrangian embedding with vanishing Maslov class of a spin manifold  $L$ in $T^*V$ must project  to a map $L \to V$ inducing an isomorphism in homology over $\QQ$.  Thus, under these hypotheses, the group $\GGC_{*,\QQ}$ is clearly trivial.

 \medskip
 The main goal of this paper is to prove the fundamental result that, when $L$ is a weakly exact Lagrangian submanifold of a symplectic manifold $M$, and under certain natural conditions on $L$ only related to the choice of the coefficients ring, this still holds, that is to say the group $\GGC_*$ is trivial. Thus, at least as far as the group $\GGC_*$ is concerned, a weakly exact Lagrangian submanifold behaves like the zero section of a cotangent bundle. The additional natural conditions to which we referred are the usual conditions under which the Floer homology over $R$ is well-defined. We recall that, by definition, $L$ is {\em weakly exact} if
 \begin{quotation}
$I_\omega : \pi_2(M, L) \to \RR : \beta \mapsto \int_\beta \omega$
\end{quotation}
vanishes. Obviously, this implies that $M$ is symplectically aspherical, i.e that $I_\omega$ vanishes on $\pi_2(M)$.

  \medskip
 Our main result is  the following:

 \begin{theorem}\label{thm:main}  Let $(M, \omega)$ be a symplectic manifold and $L \subset M$ a closed weakly exact Lagrangian submanifold. Let $g_{t \in [0,1]}$ be a Hamiltonian diffeotopy of $M$ starting at the identity and ending at a diffeomorphism preserving $L$. Then the map on homology $g_{1*} : H_*(L;R) \to H_*(L;R)$ induced from $g_1|_L$ is the identity. In this statement, we can use $\ZZ$, $\QQ$ or $\ZZ_2$ as coefficient  ring $R$, but in the case of $\ZZ$ or $\QQ$-coefficients, we require $L$ to be orientable, relatively Spin and $g_1|_L$ to preserve the orientation of $L$.

\end{theorem}
\noindent

  \medskip
     A priori, one could try to prove this statement by  using the invariance of the Floer homology along the flow induced by $g_{t \in [0,1]}$ to extract the informations on $\GGC_*$, a bit like what one would do using the sheaf of Floer homologies of a given Lagrangian submanifold induced by some Lagrangian fibration. This paper shows how to make sense of this idea. Indeed, our
     approach  relies on  our Relative Seidel morphism, introduced in Hu-Lalonde \cite{HuLalonde}, associated to a Hamiltonian path $g_{t \in [0,1]}$ with $g_1 \in \Ham_L(M)$, and on the equivalence between two versions of this morphism, one given in analytical terms and the other in geometric terms. The geometric version of this morphism can be considered as the right set up for an implementation of the above ``sheaf approach'' to the proof of our theorem.

     The idea of the proof of the theorem is to first consider the fibration $L \hookrightarrow N \to S^1$ over $S^1$ induced by the restriction to $L$ of the path $g_{t \in [0,1]}$ of Hamiltonian diffeomorphisms.   We then get the Wang long exact sequence

     $$\ldots \to H_{q+1}(L) \xto{i_*} H_{q+1}(N) \to H_q(L) \xto {g_{1*} - \id} H_q(L) \xto {i_*} H_{q}(N) \to \ldots$$

We must therefore show that

$$i_* : H_*(L) \to H_*(N) $$

\noindent
is injective. The point is that this fibration  constitutes  the boundary condition of a Dirichlet problem for the $\bar{\partial}$-operator. Indeed, the path $g_{t \in [0,1]}$ naturally gives rise to a  relative fibration  $(P, N)$ over $(D^2, S^1)$ with fiber $(M, L)$ and the linearity of the relative Gromov-Witten invariants will lead us to a proof of the injectivity of
$ i_* : H_*(L) \to H_*(N)$. This scheme of proof can be considered as the relative Lagrangian version of the main theorem in  Lalonde-McDuff-Polterovich \cite{LMP}.  This will be shown in \S \ref{sec:proofoftheorem} using our geometric Seidel map defined in the next section.\  In the last section of the paper, we give another proof of our main theorem which is more algebraic and simpler, but less geometric. It is possible that this second proof could also be derived using the main results in \S~22 of Fukaya-Oh-Ohta-Ono \cite{FOOO}.

\medskip
Note that this theorem, in its contrapositive 
version,  provides an obstruction to the extension of a given diffeomorphism $f:L \to L$ to a Hamiltonian diffeomorphism of the ambient symplectic manifold: more precisely, if $L$ is weakly exact and $f:L \to L$ induces a map  not equal to the identity on say $H_*(L; \ZZ_2)$, then it cannot be extended to a Hamiltonian diffeomorphism of $M$.
Obviously, because a Hamiltonian diffeomorphism is isotopic to the identity,
the statement of the theorem is interesting only when $H_*(L)$ does not inject in $H_*(M)$. There are plenty of closed weakly exact manifolds whose homologies do not inject in the homology of the ambient manifold. The simplest example is the one of a simple closed curve of a closed Riemann surface that bounds homologically but not homotopically, and product of these examples. A less trivial example is the following: consider the quotient $Q$ of $T^2 \times \RR$ by the linear map $(x,y,z) \mapsto (x+y,y,z+1)$ where $x,y$ are the coordinates on $T^2$ and $z$ on $\RR$.  Then $(dx - zdy) \wedge dz$ descends to a form on $Q$. Let $t$ be the coordinate on $S^1$. The form $\omega = (dx - zdy) \wedge dz + dy \wedge dt$ is then well defined and is symplectic on $Q \times S^1$.
Moreover, $T^2$ is Lagrangian.
The Wang sequence for the mapping torus shows that the kernel of $H_1(T^2) \to H_1(Q)$ is the class $[y]$. The homotopy exact sequence of the fibration $Q \to S^1$ shows that $\pi_1(T^2) \to \pi_1(Q)$ is injective. It then follows that the relative $\pi_2(M, T^2)$ vanishes, which implies that $(M, T^2)$ is weakly exact, and the homology does not inject.
More important, this example provides an instance of a Lagrangian isotopy (around the $z$-coordinate) of a weakly exact Lagrangian 2-torus such that the time one map of the isotopy is a self-map of the 2-torus  with non-trivial monodromy (it is a Dehn twist). This shows that our main result is actually a theorem in the Hamiltonian category.

 Here is the plan of the paper: in the next section, we describe the general set up, including the definition of our geometric relative Seidel morphism. In \S~3, we give the proof of the main theorem up to the statement according to which the geometric Seidel morphism is an isomorphism which is proved in \S~4 of this paper. The last section gives the algebraic proof to which we referred above. 

\medskip
\noindent
{\bf Acknowledgments}. We are grateful to Leonid Polterovich for pointing out to us the reference \cite{Yau} by Mei-Lin Yau. We  would like to thank Doug Park for suggesting the Thurston manifold.

\section{The Geometric Relative Seidel map}
We present the geometric Seidel morphism for Lagrangian submanifolds. In particular, we show that it is well defined and that  it satisfies the properties that we need.

\subsection{Quantum homology of $L$}
We recall the definition of the linear cluster complex (or pearl complex) as described in \cite{Oh, BiranCornea, CorneaLalonde}. Let $(M, \omega)$ be a symplectic manifold and $L \subset M$ a weakly exact Lagrangian submanifold.
Let $J$ be a $\omega$-compatible almost complex structure on $M$. Then the fact that $L$ is weakly exact implies that there are no non-trivial $J$-holomorphic spheres in $M$ as well as non-trivial $J$-holomorphic discs with boundary on $L$. It follows that the quantum homology of $L$ (cf. Biran-Cornea \cite{BiranCornea}) is well defined and is isomorphic to $H_*(L) \tensor \Lambda_L$.

 In this section, we will work in a  more general setting and will only assume
that $L$ is \emph{monotone}, i.e that there is a non-negative
constant $\lambda$ such that
\begin{quotation}
$I_\omega = \lambda I_\mu$ where $I_\mu:  \pi_2(M, L) \to \RR :\beta \mapsto \mu(\beta)$
is the Maslov index of $\beta$.
\end{quotation}

Suppose that $(M, L)$ is monotone with minimal Maslov index at least $2$.
Let $f \in C^\infty(L)$ be a Morse function and $\rho$ a Riemannian metric on $L$ so that the pair $(f, \rho)$ be Morse-Smale. Consider the configurations of $J$-holomorphic discs connected by the negative flow lines of $f$. More precisely, let $p_0 = x, q_k = y \in Crit(f)$, and $u_i : (D^2, S^1; -1, 1) \to (M, L; q_{i-1}, p_i)$, $i = 1, \ldots, k$ be $J$-holomorphic discs with boundary on $L$, such that
for each pair $(p_i, q_i)$, there is an open interval $I_i = (a_i, b_i) \subset \RR$ and $l_i : I_i \to L$ such that
$$\frac{d}{dt}l_i(t) = -(\nabla_\rho f)(l_i(t)) \text{ and } \lim_{t \to a_i} l_i(t) = p_i, \lim_{t \to b_i} l_i(t) = q_i.$$
The value $d_i := b_i - a_i$ is said to be the \emph{distance} between $p_i$ and $q_i$, which are $\infty$ for $i = 0, k$.
Let $\beta_i = [u_i] \in \pi_2(M, L)$ denote the class represented by the disc $u_i$, and set
$$\beta:= (\beta_1, \ldots, \beta_{k}) \text{ and } |\beta| := \sum_{i = 1}^{k} \beta_i.$$
Let $\widetilde {\mathcal M}(M, L; \beta; f, \rho, J; x, y)$ denote the space of such configurations. We note that for each $l_i$ and $u_i$ there is a one-parameter family of reparametrization symmetries. The {\em unparametrized moduli space}  is defined to be the quotient by all such symmetries:
$${\mathcal M}(M, L; \beta; f, \rho, J; x, y) := \widetilde {\mathcal M}(M, L; \beta; f, \rho, J; x, y) / \RR^{2k+1}.$$

In order for the theory to be well defined, we need regularity and transversality assumptions, as well as assumptions so that no branching of the linear cluster is possible at dimensions $\leqslant 1$ (cf. \cite{BiranCornea}). It is well known that all of these assumptions are satisfied in the monotone case with  minimal Maslov index at least $2$, and therefore in the weakly exact case.  With all such assumptions in place, we write down the dimension of the moduli space that we have just defined. Let $\mu_L$ denote the Maslov class for $L$ and $|x|$ the Morse index of $x \in Crit(f)$, then
\begin{equation}\label{eq:modulidim}
\dim_\RR {\mathcal M}(M, L; \beta; f, \rho, J; x, y) = |x| - |y| + \mu_L(|\beta|) - 1.
\end{equation}
The moduli spaces above are not necessarily compact as they admit real codimension $1$ boundaries of three types:
\begin{enumerate}
\item breaking of a Morse flow line, i.e. $d_i \to \infty$
\item bubbling off of a holomorphic disc, i.e. $\beta_i \to \beta_i' + \beta_i''$
\item shortening of a Morse flow line, i.e. $d_i \to 0$
\end{enumerate}
Since our assumptions exclude branching of the cluster in dimension $\leqslant 1$, the bubbling off of a holomorphic disc in a linear cluster with $k$ holomorphic discs gives rise to a linear cluster with $k+1$ holomorphic discs with some $d_i = 0$. Now consider the following union of equi-dimensional moduli spaces (in low dimensions)
$${\mathcal M}(M, L; B; f, \rho, J; x, y) := \union_{|\beta| = B} {\mathcal M}(M, L; \beta; f, \rho, J; x, y).$$
Then the type $(2)$ and $(3)$ boundaries cancel each other and only the type $(1)$ boundary remains.
This provides the essential idea of the following proposition (Oh \cite{Oh}, Cornea-Lalonde \cite{CorneaLalonde} and Biran-Cornea \cite{BiranCornea}).
\begin{prop}
The \emph{linear cluster complex} (or \emph{pearl complex}) of $(M, L)$ is  given by the following differential $\partial_{Pearl}$ on $Crit(f) \tensor \Lambda_L$,
$$\partial_{Pearl} x = \sum_{y, B} \# {\mathcal M}(M, L; B; f, \rho, J; x, y) e^B y$$
where $\Lambda_L$ is the Novikov ring for $L$, and where the counting is performed for $0$-dimensional moduli spaces only. The differential satifies $\partial_{Pearl}^2 = 0$. The \emph{quantum homology} of $L$ (in $M$) is defined as
$$QH_*(M, L; f, g, J) := H_*(Crit(f) \tensor \Lambda_L, \partial_{Pearl}).$$
\end{prop}
The coefficients $\# {\mathcal M}(M, L; B; f, \rho, J; x, y)$ is the counting in $\ZZ_2$ or $\QQ$. One can always work over $\ZZ_2$, while when $L$ is relatively spin and a relative spin structure is chosen, $\QQ$-coefficients can be used.

We note that the differential can be written as the sum
$$\partial_{Pearl} = \partial_{Morse} + \partial_{Pearl}',$$
where $\partial_{Morse}$ is the classical Morse differential and we may consider the cluster complex as a deformation of the classical Morse complex.
It is shown in \cite{CorneaLalonde} and  \cite{BiranCornea} that the linear cluster complex is well defined and independent of the choices made. It is therefore an invariant of $(M, L)$. We write $QH_*(M, L)$ for its homology. By a PSS type argument, it is shown in \cite{BiranCornea} that $QH_*(M, L)$ is isomorphic to $FH_*(M, L)$, the Floer homology of connecting Hamiltonian paths, whose definition in our setting is recalled in section \S~\ref{identify} for the convenience of the reader.

Under the weakly exact assumption, we note that $\partial'_{Pearl} = 0$ since there is no pseudo-holomorphic discs representing non-trivial class. Thus, the pearl complex of $(M, L)$ is the Morse complex of $L$ with coefficients in $\Lambda_L$.

\subsection{Bundle over a disc}
Let $\Ham_L(M, \omega)$ be the subgroup of Hamiltonian diffeomorphisms that preserve $L$. Let $\mc P_L\Ham(M, \omega)$ consist of paths $g_{t \in [0,1]}$ in $\Ham(M, \omega)$ such that
$$g_0 = \id \text{ and } g_1 \in \Ham_L(M, \omega).$$
Such $g=g_{t \in [0,1]}$'s define a Hamiltonian fibration over $D^2$ as follows (a similar fibration was actually carried over in a different context in Akveld-Salamon \cite{AS})  .

We consider the (closed) unit disc $D^2$ and let $D^2_\pm = \{z \in D^2 | \pm \Re z \vargeq 0\}$ be the right and left half discs. Then the fibration defined by $g$ is
$$P_{g} = M \times D^2_+ \sqcup M \times D^2_- / \sim : (x, (1 - 2t)i) \sim (g_t(x), (1-2t)i) \text{ for } t \in [0,1].$$
Let $\pi : P_g \to D^2$ denote the projection. On this Hamiltonian bundle, let $\tau$ be the coupling form constructed from a Hamiltonian function $K$ generating $g$, then
$$\omega_g := \tau + \kappa \pi^*\omega_0$$
is a symplectic form on $P_g$.
We note that along the $S^1$-boundary, we have the restricted bundle
$$N := \sqcup_{t \in S^1} L_t$$
that is obtained as the union of  the copies of $L$ in each fiber; it is a Lagrangian submanifold of $P$.
Because $L$ is weakly exact, the comparison theorem in \cite{BiranCornea} implies that $QH_*(M, L)$ is isomorphic to $FH_*(M, L)$. Because the fibration is Hamiltonian, it is easy to see that $(P_{g}, N)$ admits sections over $(D^2, S^1)$.

A class $B \in \pi_2(P_g, N)$ is a \emph{section class} if $\pi_*(B) \in \pi_2(D^2, S^1)$ is the positive generator, with respect to the natural orientation on $D^2$. We say that $B$ is a \emph{fiber class} if $B$ is in the image of the map $\pi_2(M, L) \to \pi_2(P_g, N)$ induced from inclusion of the fiber.
We claim that
\begin{lemma}
The following sequence of homotopy groups is exact at the middle term:
$$\pi_2(M, L) \xto i \pi_2(P_g, N) \xto j \pi_2(D^2, S^1).$$
\end{lemma}
{\it Proof:}
It follows from the diagram chasing:
$$
\text{\xymatrix{
\pi_2(M) \ar[d]_k\ar[rr]^\cong && \pi_2(P_g) \ar[d] && \\
\pi_2(M, L) \ar[d]_\partial \ar[rr]^i && \pi_2(P_g, N) \ar[d]^\partial \ar[rr]^j && \pi_2(D^2, S^1) = \ZZ \ar[d]_\partial^\cong\\
\pi_1(L) \ar[d]\ar@{^{(}->}[rr] && \pi_1(N) \ar[d] \ar[rr] && \pi_1(S^1)=\ZZ \\
\pi_1(M) \ar[rr]^{\cong} && \pi_1(P_g) &&
}}$$
where the columns are homotopy exact sequences of pairs and all rows except the second are homotopy exact sequences of fibrations.

We only show that $\ker j \subset \img i$. We will use $i$ and $j$ to denote the first and second map in each row. Let $\beta \in \pi_2(P_g, N)$ and $j(\beta) = 0$. Then $j(\partial \beta) = 0$ and by exactness of the third row, we see that there exists $ \alpha' \in \pi_1(L)$ such that $i(\alpha') = \partial \beta$. By the exactness of the columns, there is $\alpha \in \pi_2(M,L)$ such that $\alpha' = \partial \alpha$ and $\partial(\beta - i(\alpha)) = 0$. Thus there is $\gamma \in \pi_2(M) \cong \pi_2(P_g)$ such that $i\circ k(\gamma) = \beta - i(\alpha)$. It follows that $\beta \in \img i$.
\qed

It means in particular that the difference of section classes is a fiber class.
\begin{defn}\label{lagSeidel:vertmaslov}
Let the smooth map $u: D^2 \to P_g$ represent $B \in \pi_2(P_g, N)$. The \emph{vertical Maslov index} of $B$, denoted $\mu^v(B)$ is the Maslov index of the bundle pair $(u^*T^vP_g, u^*T^vN)$, where $T^v = \ker d\pi$ denotes the respective vertical tangent bundles.
\end{defn}
\noindent
One can show that the above is well defined and not dependent on the choice of $u$ (e.g. \cite{HuLalonde}).
Furthermore, let $B$ and $B'$ denote two section classes, we have

\medskip
$\mu^v(B - B') = \mu_L(B-B')$ and

\smallskip
$\mu^v(B) = \mu_N(B) - 2.$

\medskip
We can introduce the following equivalence relation among the section classes
$$B \sim B'  \iff \int_{B - B'} \tau = 0 \text{ and } \mu^v(B - B') = \mu_L(B - B') = 0.$$

An alternative construction of the fibration $P_g$ is the following. Let $Q := M \times D^2$ and note that $N$ coincides with the mapping cylinder of $g^{-1}_1|_L$, i.e.
$$N \cong L \times [0,1] / (p, 0) \simeq (g^{-1}_1(p), 1).$$
We then have the inclusion $\rho_g: N \into Q$ given by
$$(p, t) \mapsto (g^{-1}_t(p), e^{2\pi i t}),$$
in the parametrization of $S^1$ by $t \in [0,1] \mapsto e^{2 \pi i t}$. The following lemma is obvious.
\begin{lemma}\label{lem:pqisom}
$(Q, N) \cong (P_g, N)$ with the respective inclusion of $N$.
\end{lemma}
{\it Proof}:
We note that $(P_g, N)$ depends only on the homotopy class of $g$ with fixed end points. First we show that $(Q, N)$ depends only on such homotopy class as well. Let $g' \sim g$ be homotopic to $g$ in $\PPC_L \Ham(M, \omega)$ with fixed end points. Thus
$$h_t =  g'_t \circ g^{-1}_t \in \Omega\Ham(M, \omega) \text{ is a contractible loop in } \Ham(M, \omega).$$
Let $\Psi^h : D^2 \to \Ham(M, \omega)$ be a homotopy of $\{h_t\}$ to $\id \in \Ham(M, \omega)$, where $\Psi^h(e^{2\pi i t}) = h_t$ and $\Psi^h(0) = \id$. Then we have
$$\Psi : Q \xto{\cong} Q : (x, z) \mapsto (\Psi^h_z(x), z), \text{ and thus } \rho_{g'} = \Psi^h \circ \rho_g : N \into Q.$$

We describe an alternative construction of $P_g$, together with the embedding of $N$. Consider the parametrization of $D^2$ as the unit disc and the map $\psi : M \times D^2_+ \to M \times D^2_+$ given by
$$\psi(x, z) = (g^{-1}_t(x), z), \text{ where } z = 1-2t + i s \in D^2_+.$$
We define $Q$ from the quotient
$$M \times D^2_+ \sqcup M \times D^2_- / \sim' : (x, 1-2t) \sim' (x, 1-2t) \text{ for } t \in [0,1].$$
Then the map
$$M \times D^2_+ \sqcup M \times D^2_- \xto{\psi \sqcup \id} M \times D^2_+ \sqcup M \times D^2_-$$
induces an isomorphism of $Q \xto\cong P_g$. The inclusions of $N$ obviously correspond.
\qed

\subsection{Definition of the Seidel map} \label{mapdefn}
We first recall the definition of the map $i_* : H_*(L) \to H_*(N)$ via Morse homology. Denoting the basis of the fibration $N$ as the set of points in the unit circle of the complex plane and  by $L_{\pm 1}$ the two fibers over $+1$ and $-1$, let $F \in C^\infty(N)$ satisfy the following:

\medskip
$F$ is a Morse function

\smallskip
$f_\pm := F|_{L_{\pm 1}}$ are Morse functions on $L$

\smallskip
$Crit(F) = Crit(f_+) \union Crit(f_-)$ and

\smallskip
$\max f_- + 1 < \min f_+$.

\medskip
We choose a metric $G$ on $N$ such that the pairs $(F, G)$ and $(f_\pm, g_\pm)$ are Morse-Smale pairs on $N$ and $L$ respectively, where $g_\pm$ are the restrictions of $G$ to $L_{\pm 1}$. Then the Morse complexes are well defined and compute the homologies of the respective manifolds. To define $i_*$ via Morse theory, we require that the pair $(F, G)$ satisfies the following:
\begin{quotation}

In a neighbourhood of the fibers over $\pm 1$, the fibration $N \to S^1$ is locally identified as a product, $L \times U_\pm$, where $\pm 1 \in U_\pm \subset S^1$.

\medskip
The restriction of $F$ to this neighbourhood is of the form $f_\pm + \varphi_\pm$, where $\varphi_\pm : U_\pm \to \RR$ is smooth with unique critical point at $\pm 1$.

\smallskip
The restriction of $G$ to this neighbourhood is a product metric.
\end{quotation}

It follows that the map induced by the inclusion of the set of critical points is an inclusion of Morse complexes:
$$i : MC_*(L; f_-, g_-) \into MC_*(N; F, G).$$
It induces the map  $i_* : H_*(L) \to H_*(N)$ in homology.

We now define the Seidel map in this setting. Suppose that $J$ is a tamed almost complex structure on $P_g$, where we may choose the symplectic structure on $P_g$ to be the pull-back of the product symplectic structure on $Q$ via the isomorphism given by Lemma \ref{lem:pqisom}. We also suppose that $J$ is \emph{compatible with the fibration}, namely:

\medskip
the projection $\pi : P_g \to D^2$ is pseudo-holomorphic, and

\smallskip
$J$ restricts to compatible almost complex structures on the fibers.

\medskip
We consider the linear clusters in $(P_g, N)$ for which exactly one of the pseudo-holomorphic discs represents a section class in $\pi_2(P_g, N)$ and all other discs represent fiber classes. Given $x_- \in Crit(f_-)$ and $y_+ \in Crit(f_+)$, this amounts to consider the moduli spaces $\MMC (P_g, N; \sigma; F, G, J; x_-, y_+)$ where $|\sigma|$ is a section class.
Let $|\cdot |^L$ be the Morse index  in $L$ and $|\cdot|^N$ that in $N$, then we have
$$|x_-|^N = |x_-|^L \text{ and } |y_+|^N = |y_+|^L + 1.$$
Then by \eqref{eq:modulidim}
$$\dim_\RR \MMC (P_g, N; \sigma; F, G, J; x_-, y_+) = |x_-|^N - |y_+|^N + \mu_N(|\sigma|) - 1 = |x_-|^L - |y_+|^L + \mu^v(|\sigma|).$$
\begin{defn}
Let $\sigma$ denote a section class in $\pi_2(P, N)$ and $\sigma_0$ a particular choice of reference section class. Then $B := \sigma - \sigma_0$ is a fiber class and the chain level \emph{geometric Seidel map} is
$$\Psi_L(g, \sigma_0) : Crit(f_-) \tensor \Lambda_{L-} \to Crit(f_+) \tensor \Lambda_{L_+}$$
$$\Psi_L(g, \sigma_0)(x_-) := \sum_{B, y_+} \#{\mathcal{M}}(P_g, N; \sigma_0 + B; F, G, J_P; x_-, y_+) e^B y_+,$$
where the coefficients counts (in $\ZZ$ or $\ZZ_2$, see remark \ref{rmk:coefficients}) the zero dimensional moduli spaces.
\end{defn}

\begin{remark}\label{rmk:coefficients}
\rm{We note that for the purpose of the main theorem, with the assumption of weak exactness on $L$, standard transversality provides that the moduli spaces above are smooth manifolds. In general, when we allow $(M, L)$ to be monotone with minimal Maslov index $2$, the transversality arguments in \cite{BiranCornea} can be adapted so that the moduli spaces are again smooth with the expected dimension in dimension $\varleq 1$. The main point in the adaptation is that the disc that represents a section class is necessarily simple.

For the counting, one can always work over $\ZZ_2$. On the other hand, if $L$ is relatively spin and the map $g_1$ preserves the chosen relative spin structure, we may use $\ZZ$ or $\QQ$ as coefficients.}
\end{remark}

\begin{lemma}

$\Psi_L(g, \sigma_0)$ is a chain map of degree $\mu^v(\sigma_0)$
and the induced map on $QH_*(M, L)$, does not depend on the choice of generic data.
\end{lemma}
{\it Proof :}
Because $(M, L)$ is weakly exact, the moduli spaces of holomorphic discs in $P_g$ with boundary in $N$, representing a section class, are compact.
Let $B$ denote a fiber class in $\pi_2(P_g, N)$. Thus, $\MMC(P_g, N; \sigma_0 + B; F, G, J_P; x_-, y_+)$ is compactified by broken Morse flow lines. From this, it follows that $\Psi_L(g, \sigma_0)$ is a chain map.
Because $\deg e^B = -\mu_L(B)$, by the dimension computation of the moduli spaces, we see that the degree of the map is $\mu^v(\sigma_0)$.

Now we show that it does not depend on the choices of generic data $(F, G, J_P)$.
Suppose that we have two triples of Morse functions, metrics and compatible almost complex structures: $(F_i, G_i, J_{P, i})$ for $i = 0, 1$. Consider the fibration $(\widetilde P_g, \widetilde N) = (P_g, N) \times [0,1]$ and endow it with a triple $(\widetilde F, \widetilde G, \widetilde J)$  where $\widetilde F$ is a smooth Morse function on $\widetilde N$,  $\widetilde G$ is a metric on $\widetilde N$ that  restricts to $G_i$ on $P_g \times \{i\}$ and where $\widetilde J$ is a smooth family of $\omega$-compatible and fibration-compatible complex structures on $P_g \times \{t\}$, which connects $J_{P, i}$ on $P \times \{i\}$. Assume that
$\widetilde F$ coincides with $\widetilde F_i := A_i F_i + C_i$ on $(P_g, N) \times \{i\}$ for some constants $A_i$ and $C_i$, $i = 0, 1$ chosen so that

$$\min f_{0, +} > \max f_{1, +} + 1 > \min f_{1, +} > \max f_{0, -} + 1 > \min f_{0, -} > \max f_{1, -} + 1.$$

 Here $f_{i,\pm}$ is the restriction of $\widetilde F$ to the fiber $L_{i,\pm}$ of $N \times [0,1] \to S^1 \times [0,1]$ above the point $(\pm 1, i)$. Assume moreover as usual that $Crit(\widetilde F) = Crit(\widetilde F_0) \union Crit(\widetilde F_1)$.

In general, we consider the configuration of a linear cluster in $\widetilde P$ connecting $x_- \in L_{0, -}$ to $y_+ \in L_{1, +}$ via negative gradient lines and holomorphic discs $u_i : (D^2, S^1) \to (P_g, N) \times \{t_i\}$ for some $t_i \in [0,1]$, the total class of the $[u_i]$'s being a section class $\sigma \in \pi_2(\widetilde P, \widetilde N) \cong \pi_2(P, N)$. Exacly like in the definition of $\Psi_L$, we can adapt the transversality argument in \cite{BiranCornea} and see that the moduli spaces of such configurations are smooth of the expected dimension when the dimension is $\varleq 1$. Let $\mathcal{M}(\widetilde P, \widetilde N; \sigma; \widetilde F, \widetilde G, \widetilde J; x_-, y_+)$ denote the moduli space of such linear clusters and write $\sigma = \sigma_0 + B$, where $\sigma_0$ is a chosen reference section class and $B$ is a fiber class.
We define $\Sigma : Crit(\widetilde F_{0,-}) \tensor \Lambda_L \to Crit(\widetilde F_{1, +}) \tensor \Lambda_L$ by
$$\Sigma(x_{-}; \sigma_0) := \sum_{B, y_+} \# \mathcal{M}(\widetilde P, \widetilde N; \sigma_0 + B; \widetilde F, \widetilde G, \widetilde J; x_-, y_+ ) e^B y_+,$$
where the counting $\# \mathcal{M}$ is for the moduli spaces of expected dimension $0$, i.e.
$$|x_-|^L - |y_+|^L + \mu^v(\sigma_0) + \mu_L(B) + 1 = 0.$$
It follows that $\Sigma(\cdot; \sigma_0)$ is of degree $\mu^v(\sigma_0) + 1$.

Now considering the boundary components of the moduli spaces with dimension $1$, we see that it consists of the following four types of configurations, corresponding precisely to the case when the breaking of Morse flowlines happens on one of the $L_{i, \pm}$'s.

It is then obvious that:
\begin{enumerate}
\item the breaking on $L_{0, -}$ corresponds to $\Sigma(\cdot; \sigma_0) \circ \partial_{Pearl}$
\item the breaking on $L_{1, +}$ corresponds to $\partial_{Pearl} \circ \Sigma(\cdot; \sigma_0)$
\item the breaking on $L_{0, +}$ corresponds to $\Phi_{+} \circ \Psi_{L_0}(g, \sigma_0)$ and
\item the breaking on $L_{1, -}$ corresponds to $\Psi_{L_1}(g, \sigma_0) \circ \Phi_{-}$
\end{enumerate}
where
$$\Phi_{\pm} : Crit(f_{0, \pm}) \tensor \Lambda_L \to Crit(f_{1, \pm}) \tensor \Lambda_L$$
are the comparison maps between quantum homologies of $L$ with different choices of $(f, \rho, J)$, which are quasi-isomorphisms. We can then write down
$$\Sigma(\cdot; \sigma_0) \circ \partial_{Pearl} - \partial_{Pearl} \circ \Sigma(\cdot; \sigma_0) = \Phi_+ \circ \Psi_{L_0}(g, \sigma_0) - \Psi_{L_1}(g, \sigma_0) \circ \Phi_-$$
and it follows that $\Phi_+ \circ \Psi_{L_0}(g, \sigma_0)$ and $\Psi_{L_1}(g, \sigma_0) \circ \Phi_-$ induce the same maps on quantum homologies.
\qed

We note that the above lemma implies that the Seidel map $\Psi(g, \sigma_0)$ does not depend on the homotopy class of $g$, because the construction for $g' \sim g$ gives the same bundles $(P, N)$ with a different set of data $(F, G, J_P)$.

\section{Proof of the theorem}\label{sec:proofoftheorem}

We postpone to the Section 4 of this paper the proof of Proposition~\ref{prop:identification} stating that our geometric Seidel morphism coincides with our analytical Seidel morphism defined in \cite{HuLalonde} via a PSS-isomorphism. Since the analytical morphism is evidently an isomorphism by contruction, we get the following corollary that we will use in the present section:

\begin{corollary}\label{cor:iso}
The map $\Psi_L(g, \sigma_0)$ is an isomorphism of quantum homology of $(M, L)$.
\qed
\end{corollary}

As shown in the introduction, the main theorem (Theorem \ref{thm:main}) follows from
\begin{lemma}
Let $(M, L)$ be symplectically aspherical. Then $i_* : H_*(L) \to H_*(N)$ is injective.
\end{lemma}
{\it Proof: }
Note that the Morse theoretical definition of $i_*$ is defined by the inclusion of chain complexes $MC_*(L) \into MC_*(N)$. Since $(M, L)$ is symplectically aspherical, $QH_*(M, L) = H_*(L; \Lambda_L)$ is the homology of $(MC_*(L) \tensor \Lambda_L, \partial_{Morse})$. The Seidel map $\Psi_L(g, \sigma_0)$ then defines an isomorphism of $QH_*(M, L)$.

By the universal coefficient theorem, the group $H_*(L) \tensor \Lambda_L$ is a subgroup of $QH_*(M, L)$.
We prove by contradiction. Suppose that $\ker i_* \neq \{0\}$, then there exists $\alpha \neq 0 \in \ker i_*$ such that it is represented by $\sum_i a_i x_{i, -} \in MC_*(L_-)$ and
$$\sum_i a_i x_{i, -} = \partial_{Morse}^N \sum_j b_j y_{j, +} \text{ for some } y_{j,+} \in Crit(f_+),$$
where $\partial_{Morse}^N$ denote the boundary operator in Morse homology of $N$. Since $\Psi_L$ is an isomorphism on $QH_*(M, L)$, we have $\Psi_L(\sum_i a_i x_{i, -}) \neq 0$.

We work on chain level. Let $y \in Crit(F)$ and $z_+ \in Crit(f_+) \subset Crit(F)$ and consider the moduli space for fiber classes $B$:
$${\MMC}(P_g, N; \sigma_0 + B; F, G, J_P; y, z_+),$$
which has expected dimension
$$|y|^N - |z_+|^N + \mu_N(\sigma_0 + B) - 1 = |y|^N - |z_+|^L + \mu^v(\sigma_0) + \mu_L(B).$$
Let $\Sigma(g, \sigma_0)$ denote the map $MC_*(N) \mapsto MC_*(L) \tensor \Lambda_L$ defined by
$$y \mapsto \sum_{B, z_+} \#{\MMC}(P_g, N; \sigma_0 + B; F, G, J_P; y, z_+) e^B z_+,$$
where the coefficients count dimension $0$ moduli spaces.
We note that when restricted to $MC_*(L_-)$, $\Sigma(g, \sigma_0)$ coincides with $\Psi_L$.

Consider now the moduli spaces ${\MMC}(P_g, N; \sigma_0 + B; F, G, J_P; y, z_+)$ of dimension $1$,
which are compactified by the breakings in the Morse flowlines in $N$, on $L_\pm$. By the choices made for $(F, G)$, we write down the boundary components for $y = y_+$:
\begin{itemize}
\item ${\MMC}_{Morse}(y_+, y'_+) \times {\MMC}(P_g, N; \sigma_0 + B; F, G, J_P; y'_+, z_+)$
\item ${\MMC}_{Morse}(y_+, x_-) \times {\MMC}(P_g, N; \sigma_0 + B; F, G, J_P; x_-, z_+)$
\end{itemize}
where ${\MMC}_{Morse}$ denotes the moduli space of Morse trajectories (in $N$) connecting the two critical points. It follows that
$$\Sigma(g, \sigma_0)\circ \partial_{Morse}^N(y_+) = 0.$$
Thus
$$\Psi_L(\sum_i a_i x_{i, -}) = \Sigma(g, \sigma_0)(\sum_i a_i x_{i, -}) = \sum_j b_j \Sigma(g, \sigma_0)\circ \partial_{Morse}^N(y_{j, +}) = 0.$$
This is a contradiction.
\qed

\section{Correspondence between the analytic and geometric Seidel maps} \label{identify}
We first recall the construction of the analytic Seidel map and restate it in the current geometric setting. Then we show that the two constructions coincide. In a way similar to Seidel \cite{Seidel}, the comparison uses the PSS isomorphism in the Lagrangian setting, for which we will adapt the construction of Biran-Cornea \cite{BiranCornea} or Hu-Lalonde \cite{HuLalonde} using Hamiltonian fibrations.

For the purpose of this section, we reparametrize the half discs $D_\pm^2$ as
$$D_\pm^2 = \left\{z \in \bb C : \left|z - \frac{i}{2} \right| \varleq \frac{1}{2} \text{ and } \pm \Re z > 0 \right\},$$
and let
$\partial_0 = D^2_+\inter i \RR$ and $\partial_+ = D^2_+ \inter \{z : |z - \frac{i}{2}| = \frac{1}{2}\}$.
Let $\mathcal P_LM$ be the space of contractible paths in $M$ with both ends on $L$ and $\tilde{\mathcal P}_LM$ the covering space whose elements are equivalent classes $[l, w]$ of pairs $(l, w)$
$$l: ([0,1],\{0,1\}) \to (M, L) \text{ and } w: (D^2_+, \partial_0, \partial_+) \to (M, L, l),$$
where $l(t) = w(i(1-t))$.
The equivalence relation is the following
$$(l, w) \sim (l', w') \iff l = l' \text{ and } w \sim_{\partial_+} w'.$$
The Floer homology of $(M, L)$ is constructed from a choice of a time-dependent Hamiltonian function $H$ and a compatible almost complex structure $J$ on $M$. The action functional as well as the metric are then defined as:
$$a_{H}([l, w]) = -\int_{D^2_+}w^*\omega + \int_{[0,1]} H_t(l(t)) dt,$$
and
$$(\xi, \eta)_{J} = \int_{[0,1]} \omega(\xi(t), J_t\eta(t)) dt, \text{ for } \xi, \eta \in C^\infty(l^*TM).$$
Then the Floer homology $FH_*(M, L; H, J)$ can be seen as the Morse homology of $\tilde P_LM$ for $a_H$ in the metric $(,)_J$. The differential $da_H$ and $(,)_J$ are well defined on $P_LM$ already and the equation of negative gradient flow lines can be written in $P_L M$ as well:
\begin{equation}\label{Floerequation}
\left\{\begin{matrix}
\frac{\partial u}{\partial s} + J_t(u)
\left(\frac{\partial u}{\partial t} - X_{H_t}(u)\right) = 0 & \text{ for all }
(s, t) \in {\mathbb R} \times [0,1], \\
u|_{{\mathbb R} \times \{0, 1\}} \subset L
\end{matrix}\right.
\end{equation}


The action of $\mc P_L\Ham(M, \omega)$ on $\mc P_L M$ lifts to an action of an extension group $\tilde {\mc P}_L\Ham(M, \omega)$ on $\tilde{\mc P}_L M$. An element of $\tilde {\mc P}_L\Ham(M, \omega)$ can be represented as $(g, \tilde g)$, where $g \in \mc P_L\Ham(M, \omega)$ and $\tilde g$ is determined by the action of $(g, \tilde g)$ on point elements $[p, p] \in \tilde{\mc P}_L M$ in the following fashion. Let's denote $(g, \tilde g)\circ [p, p]$ by $[p^g, w_p^g]$. Let $[l, w] \in \tilde{\mc P}_LM$ and $[l^g, w^g] = (g, \tilde g) \circ [l, w]$, then
$$l^g(t) = g_t \circ l(t)$$
and $w^g$ is defined as follows (we write $w^g$ instead of $w^{(g, \tilde g)}$ to simplify the notation). Consider $w$ as a homotopy from a constant path $l(0)$ to $l$, i.e. it is spanned by a one-parameter family $\alpha_{\tau \in [\frac{1}{2}, 1]}$ with $\alpha_{\frac{1}{2}}(t) = l(0)$ and $\alpha_1(t) = l(t)$. The action of $g$ on $w$ is then the strip obtained as the image $g_t(\alpha_\tau(t))$. Now we define the action of $(g, \tilde g)$ on $w$ as the half disk obtained by gluing the above strip along its boundary at $\tau = \frac{1}{2}$ with $w_{l(0)}^g$. This defines $w^g$.

Now  the push-forward by $g$ of $H$ and  $J$ is given by the pair $(H^g, J^g)$:
$$H^g(t, x) = H(t, g_t^{-1}(x)) + K(t, x) \text{ and } J^g_t = dg_t \circ J_t \circ dg^{-1}_t.$$
where $K(t, x)$ is the Hamiltonian function generating $g$. The lifted action of $(g, \tilde g)$ on $\tilde P_LM$ then defines an isomorphism of Floer homologies
$$\Psi_{\tilde g} : FH_*(M, L; H, J) \to FH_*(M, L; H^g, J^g) : [l, w] \mapsto [l^g, w^g]$$
which gives the relative Seidel map.

We consider the PSS isomorphism in the fibration setting. Let $Z_-$ be the half-disc with infinite end:
$$Z_- := D^2_- \union_{\partial_0} (\RR^+ \times i[0, 1]) \subset \CC,$$
and choose a $C^\infty$  $z$-dependent ($z \in \Sigma_-$) Hamiltonian function and a compatible almost complex structure $(\mb H, \mb J)$ on $M$ such that
\begin{itemize}
\item $(\mb H, \mb J)|_{z \in D^2_-} = (0, J_0)$ and
\item $(\mb H, \mb J)|_{\Re z > 1} = (H, J)$,
\end{itemize}
where $J_0$ is a generic compatible almost complex structure.
The equation is
\begin{equation}\label{PSSequation}
\left\{\begin{matrix}
\frac{\partial u}{\partial s} + J_z(u)
\left(\frac{\partial u}{\partial t} - X_{H_z}(u)\right) = 0 & \text{ for all }
z = s + it \in Z_-, \\
u|_{\partial Z_-} \subset L
\end{matrix}\right.
\end{equation}
Then finite energy solutions converge to critical points $l$ of $da_H$ as $s \to \infty$. We may mark the point $\frac{1}{2}(1, -i)$ and consider the evaluation map from the moduli space of the solutions to \eqref{PSSequation}. The PSS isomorphism $QH_* \to FH_*$ is then defined by counting the intersections with cycles in $L_- := L \times \left\{\frac{1}{2}(-1, i)\right\}$ of the moduli space under the evaluation.

The equation \eqref{PSSequation} can be written also as the $\bar\partial$-equation for a holomorphic section in the fibration over $Z_-$ as follows.
Let $P_- = M \times Z_-$ and consider the symplectic form:
$$\Omega_- := \kappa(\omega+ d H_z \wedge dt) + ds \wedge dt.$$
Associated to this symplectic structure, the \emph{symplectic connection} is given by
$$Hor(x, z) = \Span_{(x, z)}\left(\frac{\partial}{\partial s}, \frac{\partial}{\partial t} - X_{H_z}\right).$$
The almost complex structure $\tilde J_-$ on $P_-$ is given by $J_z$ along the fibers and
$$\tilde J_-\left(\frac{\partial}{\partial s}\right) = \frac{\partial}{\partial t} - X_{H_z}.$$
Then the $\bar\partial$-equation for $\tilde J_-$-holomorphic sections $\sigma_- : Z_- \to P_-$ with boundary on $L \times \partial Z_-$ coincides with \eqref{PSSequation} and the graph of a solution of \eqref{PSSequation} gives a $\tilde J_-$-holomorphic section. Thus we obtain the geometric version of the PSS isomorphism in one direction.

Let $\tau : \CC \to \CC$ be the anti-linear map reversing the real part:
$$\tau(s + it) = -s + it.$$
Define $D_+ = \tau(D_-)$, $Z_+ = \tau(Z_-)$ and $P_+ = M \times Z_+$.
We first carry out the PSS fibration construction also for $(H^g, J^g)$, obtaining on $P_+$ an almost complex structure $\tilde J_+^g$ tamed by the symplectic form $\Omega_+^g$, where
$$\Omega_+^g = \kappa (\omega + dH^g_z \wedge dt) + ds\wedge dt \text{ when } s < -1,$$
and equals to the product form over $D^2_+$.
In this case, pseudoholomorphic sections with finite energy converge to critical points $l^g$ of $da_{H^g}$ as $s \to -\infty$. Then counting of moduli spaces of sections gives the other direction $FH_* \to QH_*$ of the PSS isomorphism.

Now the statement we will show is
\begin{prop}\label{prop:identification}
The composition $$QH_*(M, L) \xto{PSS} FH_*(M, L) \xto{\Psi_{\tilde g}} FH_*(M, L) \xto{PSS} QH_*(M, L)$$ coincides with the geometric Seidel's map $\Psi_L(g, \sigma_0)$ for an appropriately chosen $\sigma_0$.
\end{prop}
{\it Proof: } The proof is an application of the gluing method to explicitly identify the moduli spaces involved. Note that this only concerns the moduli spaces of dimension $\varleq 1$. It is analogous to the proof showing that PSS maps are isomorphisms (cf. \cite{BiranCornea}).
%
Recall that $\Psi_{\tilde g}$ is induced from the geometric map
$$g: [0,1] \times M \to [0,1] \times M : (t, x) \mapsto (t, g_t(x)).$$
Consider for each $R > 0$ the map $G_R : [1, R + 1] \times [-1, 1] \times M \to [-R - 1, -1] \times [-1, 1] \times M$ between subsets of $P_-$ and $P_+$ given by
$$G_R(s, t, x) = (s - R - 2, g(t, x)).$$
It is then  straightforward to check that $G_{R*} (\Omega_-) = \Omega_+^g$. It follows that $P_\pm$ may be glued symplectically using $G_{R}$:
$$P_g(R) := \left\{(P_- \setminus (R+1, \infty) \times M) \union (P_+ \setminus (-\infty, -R-1) \times M) \right\}/ G_R,$$
and we obtain a Hamiltonian fibration, which is homotopic to $P_g$ given in the previous geometric construction. The symplectic form $\Omega(R)$ on $P(R)$ can be written as
$$\Omega(R) = \kappa \tau(R) + \omega_{D^2}$$
where $\tau(R)$ is a coupling form.

The almost complex structures on $P_\pm$ are glued by $G_R$ as well, and the result is denoted $\tilde J_g(R)$. For pseudoholomorphic sections $u_\pm$ in $P_\pm$ with finite energy, which converges to $[l^g, w^g]$ and $[l, w]$ respectively, their limits at $s \to \pm \infty$ are identified naturally by $g$. It follows that under gluing by $G_R$, $u_\pm$ give rise to a pseudoholomorphic section of $P_g(R)$, for $R$ big enough, using standard gluing argument.
Let $\MMC(P_+, l^g)$ and $\MMC(P_-, l)$ be the moduli spaces of pseudoholomorphic sections with prescribed limits at infinity, and $\MMC(P(R))$ the moduli space of pseudoholomorphic sections. The same gluing argument (together with compactness) shows that for all $R > 0$,
$$\MMC(P(R)) \cong \union_{l}\MMC(P_+, l^g) \times \MMC(P_-, l).$$

Next we identify the maps. Let $\alpha_-$ denote a chain in $L_- $ and $\tilde l = [l, w]$. Let $\MMC(P_-, \tilde l)$ be the subspace of $\MMC(P_-, l)$ such that
$$I_\mu(w \# (-u^-)) = 0$$
where we think of $u^-$ as a map from the half disc to $M$ via the projection $P_- \to M$. Then the PSS map $QH_* \to FH_*$ is defined at the chain level as
$$PSS(\alpha_-) = \sum_{\tilde l_-} \#(ev_- \MMC(P_-, \tilde l) \pitchfork \alpha_-) \tilde l,$$
where $ev_-$ is the evaluation map
$$\MMC(P_-, \tilde l) \to L_- : u^- \mapsto u^-\left(\frac{1}{2}(1, -i)\right).$$
On the other hand, let $\alpha_+$ be a chain in $L_+ = L \times \left\{\frac{1}{2}(1, i)\right\}$ and $\MMC(P_+, \tilde l^g)$ be the subspace of $\MMC(P_+, l^g)$ such that
$$I_\mu(w^g \# u^+) = 0.$$
Then the PSS map $FH_* \to QH_*$ is defined as
$$PSS(\tilde l^g) = \sum_{B \in \pi_2(M, L), \alpha_+} \#(ev_+ \MMC(P_+, \tilde l^g \# B) \pitchfork \alpha_+) e^B\alpha_+.$$
It follows that the composition in the statement is defined by
$$PSS \circ \Psi_{\tilde g} \circ PSS(\alpha_-) = \sum_{\tilde l} \#(ev_- \MMC(P_-, \tilde l) \pitchfork \alpha_-) PSS(\tilde l^g)$$
$$= \sum_{\tilde l, B, \alpha_+} \#(ev_- \MMC(P_-, \tilde l) \pitchfork \alpha_-) \#(ev_+ \MMC(P_+, \tilde l^g \# B) \pitchfork \alpha_+) e^B\alpha_+$$
$$ = \sum_{\alpha_+, B}\#((ev_- \times ev_+)\MMC(P(R), \sigma^g_w \# B) \pitchfork (\alpha_- \times \alpha_+)) e^B \alpha_+.$$
where the last equality uses the isomorphism of the moduli spaces via gluing, where $\sigma^g_w$ can be formally written as $w \# (-w^g)$. Here the notation $\MMC(P(R), \sigma)$ means  the subset of $\MMC(P(R))$ such that $u(R)$ lies in the section class $\sigma$.

Comparing with the expression for $\Psi_L(g, \sigma_0)$, we see that the only thing left to show is that $w\# (-w^g)$ belongs to the same section class for any $\tilde l = [l, w]$. Then take $\sigma_0$ to be the common class and the proposition is established. This we show in the next lemma.
\qed

\begin{lemma}
Let $\tilde l = [l, w] \in \tilde P_LM$ and $\tilde g \in \tilde{\PPC}_L\Ham(M, \omega)$. The equivalence class of the section $\sigma^g_w = w \# (-w^g)$ of $P(R)$ defined above does not depend on $w$.
\end{lemma}
{\it Proof: }
We show that the sections $\sigma^g_w$ are homotopic, which implies equivalence.
Suppose that $(g, \tilde g) = id$ then the resulting bundle $(P_g, N)$ is the trivial bundle pair. Let $pr: (P_g, N) \to (M, L)$ be the projection to the fiber, then $pr (w \# (-w))$ represents $0$ in $\pi_2(M, L)$. In particular, $w \# (-w)$ is homotopic to the section $\{l(0)\} \times D^2$. Thus, the equivalent class of $w \# (-w^g)$ does not depend on $w$.

Let $\tilde p = [p, p]$ be the trivial path and lifting to $\tilde P_LM$ where $p \in L$. Then $\sigma^g_p = p \# (-w_p^g)$.
Choose a path $\{p_t\}$ on $L$ which connects $p$ and $p'$, then it is obvious that $\sigma^g_p$ and $\sigma^g_{p'}$ are homotopic through sections with boundary on $N$ by $\sigma^g_{p_t}$. It follows that the equivalent class of $\sigma^g_p$ does not depend on the point $p$.

For general $[l, w]$ we show that it is the combination of the above two special cases. First, $(g, \tilde g)$ may be reparametrized such that $g_t = g_1$ for $t > \frac{1}{3}$, and
$w_{p}^g$ maps $D_+ \inter \{\Im z < \frac{2}{3}\}$ to $g_1(p)$ for all $p \in L$. Then, $w$ may be reparametrized such that $w = l(0)$ on $D_+ \inter \{\Im z > \frac{1}{3}\}$. The resulting section $\sigma_w^g$ is homotopic to the original $w \# (-w^g)$.

We can now identify the section $\sigma_w^g$ with the special cases. For $\Im z > \frac{2}{3}$, the section $\sigma^g_w$ coincides with $\sigma^g_{l(0)}$ and for $\Im z < \frac{1}{3}$, the section $\sigma^g_w$ coincides with $g_1\circ (w \# (-w))$. For $\frac{1}{3} < \Im z < \frac{2}{3}$, the section $\sigma_w^g$ coincides with $g_1(l(0))$.

Similar construction for another element $[l', w']$ gives a section $\sigma_{w'}^g$. Now the homotopies described in the special cases define a homotopy between the sections $\sigma_w^g$ and $\sigma_{w'}^g$.
%
\qed

\section{A second proof of Theorem \ref{thm:main}}

We now present another proof of the main theorem based on the \textit{analytic} Seidel map which is proved to be trivial in the symplectically aspherical case, thanks to additional algebraic structures and ideas appearing in Leclercq \cite{Leclercq}. This algebraic proof, which is geometrically less meaningful, is conceptually more elementary. We present it in the particular case when the coefficient ring $\mc R=\bb Z_2$ and the pair $(M,L)$ is symplectically aspherical since the proofs of the intermediate steps are fully written in this case in \cite{Leclercq}. The extension to the most general case of the theorem is expected to be straightforward.

\medskip
Recall that since we work under the symplectic asphericity condition, the quantum homology of $L$ is actually its Morse homology and we may forget about the Novikov ring on the Floer side.

\medskip
First, notice that there is an ``instantaneous'' version of  the analytic Seidel map described in Section \ref{identify}, consisting in using, instead of an isotopy $g\in \mc P_L \Ham(M,\omega)$, only a Hamiltonian diffeomorphism preserving $L$, $g_1\in \Ham_L(M,\omega)$. We let $\nat^i(g_1)$ denote this instantaneous version, as well as its Morse counterpart. 

Indeed, for any diffeomorphism $h\in\Diff(L)$, we can consider the morphism $\nat^i(h)$ identifying the complexes $\cm_*(L;f,\rho)$ and $\cm_*(L;f^{h},\rho^{h})$, where $f^{h}=f\circ h^{-1}$ and $\rho^{h}=(h^{-1})^*\rho$, via the following equivalences:
 \begin{align*}
   x \in \mathrm{Crit_k}(f) & \Leftrightarrow   x^h=h(x) \in \mathrm{Crit_k}(f^h)\\
 \gamma \mbox{ flow line of } (f,\rho) & \Leftrightarrow \gamma^h = h\circ \gamma \mbox{ flow line of } (f^{h},\rho^{h})
 \end{align*}
This identification induces an isomorphism on homology which commutes with the usual comparison morphism. (The commutativity can be obtained at the chain level, by choosing a regular homotopy $(\boldsymbol{f},\boldsymbol{\rho})$ between $(f_0,\rho_0)$ and $(f_1,\rho_1)$ and on the other side $(\boldsymbol{f}^{h},\boldsymbol\rho^{h})$.) The action of $h$ on the (Morse) homology of $L$ can then be seen as the composition of this identification with the usual Morse comparison  morphism (in order to end up in the initial complex $\cm_*(L;f,\rho)$). 

Now if we choose $g_1\in \Ham_L(M,\omega)$, we can compare the Floer and Morse instantaneous Seidel maps respectively associated to $g_1$ and its restriction to $L$. The first step of this second proof is to show that they do coincide, via the Lagrangian PSS morphism.
\begin{lemma}\label{lemm:1}
The following diagram commutes in homology:
\begin{align}\label{eq:diagram1}
 \begin{split}
  \xymatrix{\relax
    \cm_*(L;F,\rho) \ar[r]_{\hspace{-.3cm}\nat^i(g_{1|L})}\ar[d]_{\pss}  \ar@/^1.5pc/[rr]^{(g_1)_*} & \cm_*(L;F^{g_1},\rho^{g_1}) \ar[d]^{\pss} \ar[r]_{\hspace{.2cm}\comp}  & \cm_*(L;F,\rho) \ar[d]^{\pss}\\
    \cf_*(L;H,J) \ar[r]_{\hspace{-.3cm}\nat^i(g_1)} & \cf_*(L;H^{g_1},J^{g_1}) \ar[r]_{\hspace{.2cm}\comp}& \cf_*(L;H,J)
  }
 \end{split}
\end{align}  
\end{lemma}

\begin{proof}
Since the PSS morphism commutes with the classical Morse and Floer comparison morphisms  (see e.g \cite[Lemma 2.3]{Leclercq} for a proof), the right square of \eqref{eq:diagram1} commutes in homology. 

Now the left square commutes (at the chain level) for any regular choices of Floer and Morse data on the left and their respective pullbacks via the Hamiltonian diffeomorphism and its restriction to $L$. Indeed, in that case, there is an identification of the moduli spaces 
defining the involved PSS morphisms.
\end{proof}

Since \eqref{eq:diagram1} commutes, it suffices to prove that the composition of the two morphisms at the Floer level induces the identity in homology. In order to do so, we express the action of $\nat^i(g_1)$ on $\hf_*(L)$ in terms of the (non-instantaneous) analytic Seidel map 
(and Poincar\'e duality).
\begin{lemma}\label{lemm:2}
The following diagram commutes:
\begin{align}\label{eq:diagram2}
 \begin{split}
\xymatrix@C=15pt{\relax
    \cf_*(L;H,J) \ar[dr]_{\nat(g)}\ar[rr]^{\nat^i(g_1)} &                                                                     & \cf_*(L;H^{g_1},J^{g_1}) \ar[r]^{\mathrm{PD}} &  \cf^{\hat *}(L;\widehat{H^{g_1}},\widehat{J^{g_1}}) \\
                                                                 & \cf_*(L;H^{g},J^{g}) \ar@{..>}[ru] \ar[r]^{\mathrm{PD}} & \cf^{\hat *}L;\widehat{H^{g}},\widehat{J^{g}}) \ar[ru]_{\nat(g')}    &
}
 \end{split}
\end{align}
where  $g'$ is the isotopy defined as $g'_t=g_1 \circ g^{-1}_{1-t}$. (The doted arrow, usually denoted $[\nat(g')]_!$, is defined by the commutativity of the right square, that is, by pre- and post-composing $\nat(g')$ with Poincar\'e duality.)
\end{lemma}

Let us briefly recall the Floer theoretic version of Poincar\'e duality. This isomorphism is defined by identifying the complexes $\cf(L;H,J)$ and $\cf(L;\widehat{H},\widehat{J})$ for any regular pair $(H,J)$, where $(\widehat{H},\widehat{J})$ is ``dual'' to $(H,J)$, that is, is defined as $\widehat{H}_t(x)=-H_{1-t}(x)$ and $\widehat{J}_t=J_{1-t}$. The generators are geometrically the same orbits but considered with the opposite orientation, and so are the half-tubes defining the respective differentials. (A priori $\hat *=-*$, if the references of the Maslov indices are chosen so that they geometrically coincide. Other choices amount to global shifts of the degree which do not matter here.)

Via straightforward computations, it is easy to see that $(\widehat{H^g})^{g'}=\widehat{H^{g_1}} = \widehat{H}^{g_1}$ and $(\widehat{J^g})^{g'}=\widehat{J^{g_1}} = \widehat{J}^{g_1}$, and that all the involved pairs are regular if and only if $(H,J)$ is. Thus, \eqref{eq:diagram2} makes sense as it is (on the complexes).

\begin{remark}\label{rema:PDcompcommute}
Notice for later use, that Poincar\'e duality commutes with the usual comparison morphism of Floer homology, 
since it even commutes at the chain level as soon as one uses timewise dual homotopies.  
\end{remark}

\begin{proof}[Proof of Lemma \ref{lemm:2}]
Figure \ref{fig:diagram2} illustrates the evolution of an orbit along \eqref{eq:diagram2}.
\begin{figger}
  \centering
  \includegraphics{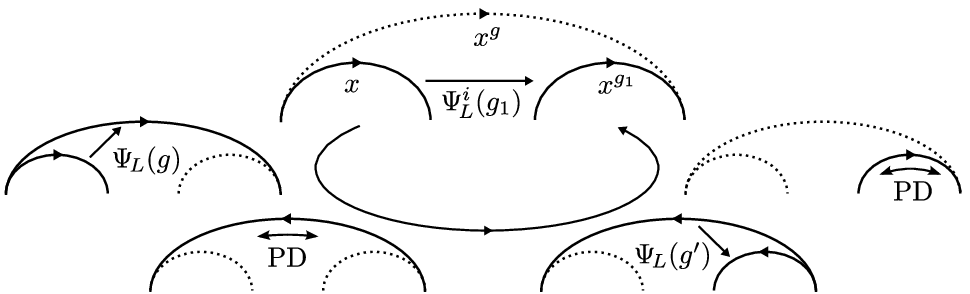}
    \label{fig:diagram2}
\end{figger}

\noindent Namely, for each time $t$, we have:
\begin{align*}
x(t) &\;\leadsto\; x^g(t) = g_t(x(t)) \;\leadsto\; \mathrm{PD}(x^g)(t) = g_{1-t}(x(1-t)) \\
& \;\leadsto\; \big(\mathrm{PD}(x^g)\big)^{g'}(t) = g_{1}(x(1-t)) \;\leadsto\; \mathrm{PD}^{-1}\big[\big(\mathrm{PD}(x^g)\big)^{g'}\big](t) = g_{1}(x(t))=x^{g_1} (t)
\end{align*}
The half-tubes defining Floer differential evolve in a similar way and the commutativity of the diagram immediately follows. 
\end{proof}

Now comes the crucial point: When $(M,L)$ is symplectically aspherical, the analytic Seidel map acts trivially on (Morse) homology.
\begin{lemma}\label{lemm:3}
The analytic Seidel map (of Proposition \ref{prop:identification}), is the identity, that is, $\nat(g)$ acts on Floer homology as the usual comparison morphism.
\end{lemma}

\begin{remark}
Recall that the PSS morphism and $\nat(g)$ can be composed since the action of $g$ preserves the component of $\mc P(L)$ consisting of contractible (in $\pi_1(M,L$)) paths (see Bialy--Polterovich \cite[Theorem 1.8]{BialyPolterovich} for the proof of this result in the weakly exact case).
\end{remark}

The proof of the theorem now easily follows. Indeed, Lemma \ref{lemm:3} allows us to replace the Seidel map by the usual comparison morphism, in the diagram induced in homology by (\ref{eq:diagram2}). Thus, we get the commutative diagram
\begin{align}\label{eq:diagram3}
 \begin{split}
\xymatrix{\relax
    \hf_*(L) \ar[d]_{\comp}\ar[rr]^{\nat^i(g_1)} && \hf_*(L) &  \\
   \hf_*(L) \ar[r]_{\mathrm{PD}} & \hf^*(L) \ar[r]_{\comp}  & \hf^*(L) \ar[u]_{\mathrm{PD}^{-1}} 
}
 \end{split}
\end{align}

Now, as noticed above (see Remark \ref{rema:PDcompcommute}), the comparison morphism commutes with Poincar\'e duality, thus we can permute the two morphisms composing the bottom line of \eqref{eq:diagram3}. Since the Floer comparison morphism is natural, this shows that $\nat^i(g_1)$ acts on $\hf_*(L)$ as the comparison morphism and this, in turn, proves that the bottom line of (\ref{eq:diagram1}) induces the identity in homology, which concludes the proof of Theorem \ref{thm:main}.\\

Therefore, it only remains to prove Lemma \ref{lemm:3}. The fact that, for a symplectically aspherical pair $(M,L)$, $\Psi_L (g)$ acts on Floer homology as the comparison morphism can be indirectly deduced from the commutative diagram of \cite[Proposition 3.1]{Leclercq}, by first arbitrarily ``cutting'' in two parts the Hamiltonian isotopy which we consider. However, the proof of \cite[Proposition 3.1]{Leclercq} itself can be easily adapted to immediately show that the diagram
\begin{align}\label{eq:diag08}
\begin{split}
\xymatrix{\relax
    \cm_*(L;f,\rho) \ar[d]_{\pss}\ar[r]^{\hspace{-.3cm}\pss} & \cf_*(L;H^g,J^g) \\
    \cf_*(L;H,J) \ar[ru]_{\nat(g)} & 
}
\end{split}
\end{align}
commutes in homology, and this commutativity amounts to the triviality of the analytic Seidel map (compare with Proposition \ref{prop:identification}).

\begin{proof}[Proof of Lemma \ref{lemm:3}]
We know from \cite[\S 3.1-3.2]{Leclercq} that all the involved groups are $\hm_*(L)$--modules, and that this additional structure is preserved by the PSS morphism and by $\nat (g)$. Thus $\Phi = (\pss)^{-1} \circ \nat(g) \circ \pss$ is an endomorphism of $\hm_*(L)$ (as a module over itself). 

Now $\Phi([L])= [L]$, since $[L]$ generates $\hm_{\mathrm{top}}(L)$ (due to $\bb Z_2$ coefficients, signs are arbitrary). Since $[L]$ is also the unit of the {\it ring} $(\hm_*(L),\,\cdot\,)$, we have for any $a\in \hm_*(L)$: 
\begin{align*}
\Phi(a) &=\Phi(a\cdot [L]) = a\cdot \Phi([L]) =a\cdot [L] =a \;.
\end{align*}
Thus $\Phi$ is the identity and \eqref{eq:diag08} commutes. Now, since the PSS morphism commutes with the usual comparison morphisms, the composition of the PSS morphisms in \eqref{eq:diag08} is nothing but the Floer comparison morphism. Thus both morphisms
$$\nat(g),\,\comp \co \xymatrix{\relax     \cf_*(L;H,J) \ar@<.5ex>[r] \ar@<-.5ex>[r]  &  \cf_*(L;H^g,J^g)  }$$
coincide in homology. 
\end{proof}

\end{document}